\documentclass[11pt]{amsart}
\usepackage{amssymb}
\usepackage{bm}
\usepackage{hyperref}
\usepackage{graphicx}

\newtheorem{thm}{Theorem}[section]

\newtheorem{cor}[thm]{Corollary}

\newtheorem{Def}[thm]{Definition}
\newtheorem{prop}[thm]{Proposition}
\newtheorem{rem}[thm]{Remark}
\newtheorem{ex}[thm]{Example}

\newcommand{\bdfn}{\begin{Def} \rm}
\newcommand{\edfn}{\end{Def}}

\newcommand{\ra}{\rightarrow}

\newcommand{\Lgra}{\Longrightarrow}
\newcommand{\Lglra}{\Longleftrightarrow}

\newcommand{\es}{\emptyset}
\newcommand{\ci}{\subseteq}
\newcommand{\ds}{\displaystyle}

\newcommand{\e}{\varepsilon}

\newcommand{\sm}{\setminus}

\newcommand{\iy}{\infty}

\newcommand{\beqa}{\begin{eqnarray*}}
\newcommand{\eeqa}{\end{eqnarray*}}

\newcounter{cnt1}
\newcounter{cnt2}
\newcounter{cnt3}
\newcounter{cnt4}
\newcommand{\blr}{\begin{list}{$($\roman{cnt1}$)$} {\usecounter{cnt1}
 \setlength{\topsep}{0pt} \setlength{\itemsep}{0pt}}}
\newcommand{\blR}{\begin{list}{\Roman{cnt4}.\ } {\usecounter{cnt4}
 \setlength{\topsep}{0pt} \setlength{\itemsep}{0pt}}}
\newcommand{\bla}{\begin{list}{$(\alph{cnt2})$} {\usecounter{cnt2}
 \setlength{\topsep}{0pt} \setlength{\itemsep}{0pt}}}
\newcommand{\bln}{\begin{list}{$($\arabic{cnt3}$)$} {\usecounter{cnt3}
 \setlength{\topsep}{0pt} \setlength{\itemsep}{0pt}}}
\newcommand{\el}{\end{list}}

\begin{document}

\title[Hahn-Banach smoothness and related properties in Banach spaces]{On Hahn-Banach smoothness and related properties in Banach spaces}

\author[Daptari]{Soumitra Daptari}

\address{Stat–Math Division, Indian Statistical Institute, 203, B. T. Road, Kolkata 700108, India}

\email{daptarisoumitra@gmail.com 
}

\subjclass[2000]{46A22, 46B10, 46B25, 46B20, 46B45. \hfill \textbf{\today} }


\maketitle
\begin{abstract}
In this paper, we study several variants of Hahn-Banach smoothness, viz., property-$(SU)$/$(HB)$/$(wU)$, where property-$(SU)$ and property-$(HB)$ are stronger notions and property-$(wU)$ is a weaker notion of Hahn-Banach smoothness. We characterize property-$(wU)$ and property-$(HB)$. It is observed that $L_1(\mu)$ has property-$(wU)$ in $L_1(\mu,(\mathbb{R}^2,\|.\|_2))$ but it does not have property-$(U)$ in $L_1(\mu,(\mathbb{R}^2,\|.\|_2))$ for a non-atomic measure $\mu$. We derive a sufficient condition when property-$(wU)$ is equivalent to property-$(U)$ of a subspace. It is observed that these properties are separably determined. Finally, finite-dimensional and finite co-dimensional subspaces of $c_0$, $\ell_p$ ($1\leq p<\iy$) having these properties are characterized.
\end{abstract}

\section{Introduction}
One of the major cornerstones of functional analysis is the Hahn-Banach theorem. It ensures that a bounded linear functional on a subspace of a normed linear space can be extended to the entire space while preserving its norm. In the study, we investigate scenarios in which the extension is unique. 
\subsection{Notations and definitions}
By $X$, we mean a Banach space over real scalars. We consider Y to be a closed subspace of X, unless we are discussing dense linear subspaces. Closed unit ball and unit sphere of $X$ are denoted by $B_X$ and $S_X$. The sets $\{x^*\in X^*:x^*(y)=0,\,\forall\, y\in Y \}$ and $\{x^*\in X^*:\|x^*\|=\|x^*|_Y\|\}$ are denoted by $Y^\perp$ and $Y^\#$, respectively.

We begin this study with the following definition by R. R. Phelps.
\bdfn\cite[Pg. 238]{P}
A subspace $Y$ of $X$ is said to have {\it property-$(U)$} if every bounded linear functional on $Y$ has a unique norm-preserving extension over $X$.
\edfn 
It's also known as a Hahn-Banach smooth subspace if $Y$ has the property-$(U)$. A weakened variant of the preceding concept is known as weak Hahn-Banach smoothness. We designate the same as property-$(wU)$.
\bdfn\cite{LA}
A subspace $Y$ of $X$ is said to have {\it property-$(wU)$} whenever $y^*\in Y^*$ satisfying $\|y^*\|=|y^*(y_0)|$
for some $y_0\in S_Y$ has a unique norm-preserving extension to $X$. 
\edfn
A few strengthenings of property-(U) are following.  
\bdfn
Let $Y$ be a subspace of $X$ such that $Y^\perp$ is complemented in $X^*$ by a subspace $G$.
\bla
\item \cite{O}  $Y$ is said to have  {\it property-$(SU)$} in $X$ if for each $x^*\in X^*$,
$$\|x^*\|>\|y^\#\|,$$
$\text{ whenever } x^*=y^\#+y^\perp \text{ with } y^\#\in G, y^\perp(\neq 0) \in Y^{\perp}.$
\item \cite{H} $Y$ is said to have {\it property-$(HB)$} in $X$ if for each $x^*\in X^*$,
$$\|x^*\|>\|y^\#\| ~\&~ \|x^*\|\geq \|y^\perp\|,$$
$$\text{ whenever } x^*=y^\#+y^\perp \text{ with } y^\#\in G, y^\perp(\neq 0) \in Y^{\perp}.$$
\item \cite{HWW}  $Y$ is called an {\it $M$-ideal} if for each $x^*\in X^*$,
$$\|x^*\|=\|y^\#\|+\|y^\perp\|,$$ 
$\text{ whenever } x^*=y^\#+y^\perp \text{ with } y^\#\in G, y^\perp\in Y^{\perp}.$
\el
\edfn
 J. Hennefeld has introduced the $HB$-subspace in \cite{H}. We refer to the same as property-$(HB)$.

For a subspace $Y$ of $X$, we define $P_Y(x)=\{y\in Y:\|x-y\|=d(x, Y)\}$, the set of all the best approximations from $x$ to $Y$. $P_Y$ is known as a metric projection on the subspace $Y$. Note that for $x\in X$, $P_Y(x)$ may be empty. If $P_Y(x)$ is non-empty for all $x\in X\sm Y$, then $Y$ is said to be a proximinal subspace of $X$. The subspace $Y$ is said to be a Chebyshev subspace if $P_Y(x)$ is singleton for all $x\in X$. 

\subsection{A short background and motivation of this work} The investigation into the uniqueness of the Hahn-Banach extension started in the decade from 1930 to 1939.  A. E. Taylor in \cite{T},  demonstrated that if the dual space of a normed linear space is strictly convex, then every subspace has property-$(U)$ in the normed linear space. After a long period, S. R. Foguel in \cite{F} discovered that the converse is true. R. R. Phelps in \cite{P} first concentrated on an individual subspace with property-$(U)$ and established that a subspace $Y$ of $X$ has property-$(U)$ if and only if the metric projection on $Y^\perp$ is single-valued. It should be noted that the metric projection may not be linear in this scenario. The author identified specific finite-dimensional and finite co-dimensional subspaces with property-$(U)$ in $\ell_1$, $c_0$, $\ell_\infty$, $C(K)$, $L_1(\mu)$, and $L_\infty(\mu)$, where $K$ is a compact Hausdorff space and $\mu$ is a positive measure. The author in \cite{H} has introduced property-$(HB)$, which is a stronger notion of property-$(U)$ but very close to $M$-ideal. The author in \cite{O} introduced property-$(SU)$ as a concept stronger than property-$(U)$ but weaker than property-$(HB)$ after a few years. It is known that $Y$ is an $M$-ideal in $X$ if and only if $P_{Y^\perp}$ is a $L$-projection, i.e., $\|x^*\|=\|P_{Y^\perp}(x^*)\|+\|x^*-P_{Y^\perp}(x^*)\|$ for all $x^*\in X^*$ (see \cite{HWW}). The authors in \cite{DP} have addressed a longstanding gap in the characterization of properties $(wU)/(SU)/(HB)$. They have characterized these three properties in terms of metric projection (see \cite[Theorem~2.1, Theorem~3.3, Theorem~3.4]{DP}).

From an alternative viewpoint, $Y$ is an $M$-ideal in $X$ if and only if $Y^{\#}$ is a closed subspace of $X^*$ and $X^*=Y^{\#}\oplus_{\ell_1}Y^\perp$ (see \cite{HWW}). The author in \cite{O} demonstrated that $Y$ has property-$(SU)$ in $X$ if and only if $Y^{\#}$ is a closed subspace of $X^*$ and $X^*=Y^{\#}\oplus Y^\perp$. It is crucial to note that $Y^\#$ may not show linearity in $X^*$ even when $Y$ has property-$(U)$ in $X$. According to \cite{P}, $Y$ has property-$(U)$ in $X$ if and only if every element in the dual space can be uniquely expressed as a sum of two elements from $Y^{\#}$ and $Y^\perp$. The corresponding characteristics for property-$(wU)$ and property-$(HB)$ are missing. The gaps are filled in this study, these two properties are described in terms of $Y^\#$ and $Y^\perp$. We explore these properties from different angles. For finite-dimensional subspaces, property-$(wU)$ and property-$(U)$ are coinside. But to date, there are no known infinite-dimensional subspaces in classical Banach spaces that have property-$(wU)$ but do not have property-$(U)$. In this study, we cover a few such examples of subspaces in classical Banach spaces. We classify the finite-dimensional and finite co-dimensional subspaces as having property-$(SU)$ and property-$(HB)$ in classical Banach spaces, as motivated by \cite{P}.

\section{Property-$(U)$ and its alterations}
In this section, we mainly study property-$(wU)$ and property-$(HB)$. We demonstrate the characterizations of these properties. We discuss the transitivity of these properties. Finally, we address the fact that these two properties, along with property-$(U)$ and property-$(SU)$, are separably determined.
\subsection{On property-$(wU)$}
We begin this subsection with the following complete characterization of property-$(wU)$, motivated by \cite[Theorem~2.1]{LA} and \cite[Theorem~2.1]{DP}.
\begin{thm}
    Let $Y$ be a subspace of $X$. Then, the following are equivalent:
    \bln
    \item $Y$ has property-$(wU)$ in $X$.
    \item $P_{Y^{\perp}}(x^*)=0$ for $x^*\in X^*$ such that $\|x^*\|=|x^*(y)|$ for some $y\in S_Y$.
    \item For any $x_1^*,x_2^*$ in $X^*$ with $\|x_1^*\|=\|x_2^*\|=|x_1^*(y)|=|x_2^*(y)|$ for some $y\in S_Y$ and $x_1^*+x_2^*\in Y^{\perp}$, we have $x_1^*+x_2^*=0$.
    \el 
\end{thm}
\begin{proof}
    $(1)\Lglra (2)$ follows from \cite[Theorem~2.1]{DP}.

    $(1)\Lgra (3):$ Let $x_1^*,x_1^*$ in $X^*$ with $\|x_1^*\|=\|x_2^*\|=|x_1^*(y)|=|x_2^*(y)|$ for some $y\in S_Y$ and $x_1^*+x_2^*\in Y^{\perp}$. Put $y^*=x_1^*|_Y=-x_2^*|_Y$. Clearly, $y^*$ attains norm on $S_Y$ and $x_1^*,-x_2^*$ are two norm-preserving extensions of $y^*$. Thus, $x_1^*=-x_2^*$ i.e., $x_1^*+x_2^*=0$.

     $(3)\Lgra (1):$ Let $y^*\in Y^*$ such that $\|y^*\|=|y^*(y)|$ for some $y\in S_Y$ and $x_1^*$, $x_2^*$ are two norm-preserving extensions of $y^*$. Clearly, $\|x_1^*\|=\|-x_2^*\|=|x_1^*(y)|=|-x_2^*(y)|$ and $x_1^*+(-x_2^*)\in Y^{\perp}$. Thus, $x_1^*+(-x_2^*)=0$ i.e., $x_1^*=x_2^*$. Hence, $Y$ has property-$(wU)$ in $X$.
\end{proof}
 We observe that property-$(U)$ is strictly stronger than property-$(wU)$.
\begin{thm}\label{WTP1}
	Let $(\Omega, \Sigma, \mu)$ be a probability space. Let $Y$ be a subspace of $X$ and suppose that $X^*$ has the RNP. Then, $Y$ has property-$(wU)$ in $X$ if and only if $L_1(\mu, Y)$ has property-$(wU)$ in $L_1(\mu, X)$.
\end{thm}
\begin{proof}
Let $\Lambda\in L_1(\mu,Y)^*\cong L_\iy(\mu,Y^*)$ such that $\|\Lambda\|=1$ and $\Lambda(f)=1$ for some $f\in S_{L_1(\mu,Y)}$. Clearly, $\Lambda(t)(f(t))\leq \|\Lambda(t)\|\|f(t)\|$ fot $t\in \Omega$. Since $\Lambda\in S_{L_\iy(\mu,Y^*)}$, we have that $\Lambda(t)(f(t))\leq\|f(t)\|$. Now, $\|f(t)\|-\Lambda(t)(f(t))\geq 0$ and $\int_{\Omega}\|f(t)\|d\mu(t)=1=\int_{\Omega}\Lambda(t)(f(t))d\mu(t)$.  Hence, $\Lambda(t)(f(t))=\|f(t)\|$ a.e. Therefore, $\Lambda$ has a unique norm-preserving extension $\Gamma$ over $L_1(\mu,X)$, where $\Gamma(t)$ is the unique norm-preserving extension of $\Lambda(t)$ over $X$ a.e.
\end{proof}
\begin{ex}
It is clear from Theorem~\ref{WTP1} that $L_1(\mu)$ has property-$(wU)$ in $L_1(\mu,(\mathbb{R}^2,\|.\|_2))$ but $L_1(\mu)$ does not have property-$(U)$ in $L_1(\mu,(\mathbb{R}^2,\|.\|_2))$ for a non-atomic measure $\mu$ (see \cite[Example~2.3]{BR}).
\end{ex}
We now present a circumstance in the infinite-dimension cases where property-$(wU)$ is identical to property-$(U)$.

A subspace $Y$ of $X$ is said to have the $2$-ball property if for any $x,y\in X$, $r_1, r_2>0$ satisfying $\|x-y\|<r_1+r_2$ and $B[x,r_2]\cap Y\neq\es$; then, $B[y,r_1]\cap B[x,r_2]\cap Y\neq\es$. In particular, if $y\in Y$; then, $Y$ of $X$ is said to have the $1\frac{1}{2}$-ball property (see \cite{Y1}).

\begin{thm}
	Let $Y$ be a subspace of $X$ with $1\frac{1}{2}$-ball property. Then,
	$Y$ has property-$(wU)$ in $X$ if and only if $Y$ has property-$(U)$ in X.
\end{thm}
\begin{proof}
	Suppose that $Y$ has $1\frac{1}{2}$-ball property and also has property-$(wU)$ in $X$. An identical construction as stated in \cite[Theorem~4]{Y2} to prove $(ii)\Lgra (iii)$ also leads to an $f$ that is norm-attaining in the subspace and separates
	the set $A$ from $x_0$. Thus, $Y$ has $2$-ball property in $X$ and, hence, has property-$(U)$ in $X$ (see \cite{Y2}). The converse is trivial.
\end{proof}
Analog to property-$(U)$ (see \cite[Theorem~5.1]{DPR} ), all that is needed to demonstrate that a subspace has property-$(wU)$ is to make sure that the subspace has property-(wU) in each of its superspaces when it is contained as a hyperplane.

\begin{thm}
	Let $Y$, $Z$ be subspaces of $X$, where $\overline{Z}=X$. Then, $Y$ has property-$(wU)$ in $\textrm{span}\{Y\cup \{z\}\}$ for all $z\in Z$ if and only if $Y$ has property-$(wU)$ in $X$.
\end{thm}
\begin{proof}
    A similar proof of [7, Theorem 5.1] works here.
\end{proof}
Our next observations are straightforward and follow from the definition of property-$(wU)$.
\begin{prop}
Let $Z$ and $Y$ be subspaces of $X$ such that $Z\ci Y\ci X$. If $Z$ has property-$(wU)$ in $X$, then $Z$ has property-$(wU)$ in $Y$.
\end{prop}
\begin{prop}\label{TwU}
Let $Z$ and $Y$ be subspaces of $X$ such that $Z\ci Y\ci X$. If $Z$ has property-$(wU)$ in $Y$ and $Y$ has property-$(wU)$ in $X$, then $Z$ has property-$(wU)$ in $X$.
\end{prop}

\subsection{On property-$(HB)$}
We start this subsection with the following characterization of property-$(HB)$.

\begin{thm}
    Let $Y$ be a subspace of $X$. Then the following are equivalent:
    \bln
    \item $Y$ has property-$(HB)$ in $X$.
    \item $Y^{\#}$ is a closed subspace of $X^*$ and $X^*=Y^{\#}\oplus Y^{\perp}$ with $\|x^*\|\geq \|y^\perp\|$, whenever $x^*=y^\#+y^\perp$, $y^{\#}\in Y^{\#}$ and $y^\perp\in Y^\perp$.
    \el 
\end{thm}
\begin{proof}
    Let $Y$ has property-$(HB)$ in $X$. It follows from \cite{O} that $Y^{\#}$ is a closed subspace of $X^*$ and $X^*=Y^{\#}\oplus Y^{\perp}$. Let $x^*=y^\#+y^\perp$, where $y^{\#}\in Y^{\#}$ and $y^\perp\in Y^\perp$. It remains to show $\|x^*\|\geq \|y^\perp\|$. Now, $d(x^*,Y^\perp)=\|x^*|_{Y}\|=\|y^\#\|=\|x^*-y^\perp\|$. Hence, $P_{Y^\perp}(x^*)=y^\perp$. By \cite[Theorem~3.4]{DP}, we have $\|y^\perp\|=\|P_{Y^\perp}(x^*)\|\leq \|x^*\|.$

    Conversely, let $Y^{\#}$ is a closed subspace of $X^*$ and $X^*=Y^{\#}\oplus Y^{\perp}$ with $\|x^*\|\geq \|y^\perp\|$, whenever $x^*=y^\#+y^\perp$, $y^{\#}\in Y^{\#}$ and $y^\perp\in Y^\perp$. It is enough to show that $\|x^*\|>\|y^\#\|$, whenever $y^\perp \neq 0$. Clearly, $\|x^*\|\geq \|x^*|_Y\|=\|y^\#\|$. If $\|x^*\|=\|y^{\#}\|$, then $x^*$ and $y^\#$ are two distinct norm preserving extensions of $y^{\#}|_Y$ as $y^\perp\neq 0$. It raises a contradiction as it follows from \cite{O} that $Y$ has property-$(SU)$ in $X$. Hence, $\|x^*\|>\|y^\perp\|$.
\end{proof}

Let us recall \cite[Theorem~5.1]{DPR}. We may derive the corresponding analog for property-$(HB)$.
\begin{thm}\label{TPH2}
	Let $X$ be a Banach space and $Y$, $Z$ be subspaces of $X$, where $\overline{Z}=X$. Then, $Y$ has property-$(HB)$ in $\textrm{span}\{Y\cup \{z\}\}$ for all $z\in Z$ if and only if $Y$ has property-$(HB)$ in $X$.
\end{thm}
\begin{proof}
Suppose that $Y$ has property-$(HB)$ in $\textrm{span}\{Y\cup \{z\}\}$ for all $z\in Z$. By \cite[Theorem~5.1]{DPR}, $Y$ has property-$(SU)$ in $X$. We assume that $Y$ does not have property-$(HB)$ in $X$. Let $x^*\in X^*$ be such that $x^*=y^\#+y^\perp$, where $y^\#\in Y^\#$ and $y^\perp\in Y^\perp$. Also we assume that $\|x^*\|<\|y^\perp\|$. Choose $z_0\in S_Z$ such that $\|x^*\|<|y^\perp(z_0)|$. Consider $W=\textrm{span}\{Y\cup \{z_0\}\}$, then $x^*|_W\in W^*,y^\#|_W\in Y_{W^*}^\#,y^\perp|_W\in Y_{W^*}^\perp$.
Clearly, $\|x^*|_W\|\leq \|x^*\|<|y^\perp(z_0)|\leq \|y^\perp|_W\|$ and hence we get a contradiction.

The converse follows from \cite[Theorem~3.16]{DP}.
\end{proof}
Proposition~\ref{TwU} states that property-$(wU)$ is transitive. From the definition, we can see that the property-$(U)$ is also transitive. \cite[Proposition~3.15.b]{DP} conveys the same for peoperty-$(SU)$. But,  we don't know whether property-(HB) is transitive. However, we have the following observation in the direction.
\begin{prop}\label{TPH1}
Let $M$ be an $M$-summmand in $X$. Suppose that $Z$ has property-$(HB)$ in $M$. Then $Z$ has property-$(HB)$ in $X$.
\end{prop}

\begin{proof}
Since $M$ be an $M$-summmand in $X$, there exists a subspace $N$ of $X$ such that $X=M\bigoplus_{\ell_\iy} N$. Hence, we have that $X^*=M^*\bigoplus_{\ell_1} N^*$. Let $P:X^*\ra M^*$ be the linear projection such that $\|x^*\|=\|P(x^*)\|+\|x^*-P(x^*)\|$. Let $P_{Z^\perp}:M^*\ra Z^\perp$ be the metric projection. By \cite[Theorem~3.4]{DP}, $P_{Z^\perp}$ is a linear projection of norm-$1$. Clearly, $P_{Z^\perp}\circ P:X^*\ra Z^\perp$ is a norm-$1$ projection. We now show that $\|I-P_{Z^\perp}\circ P\|=1$. Let $x^*\in X^*$, then
\begin{align*}
\|(I-P_{Z^\perp}\circ P)(x^*)\|&=\|x^*-P(x^*)+P(x^*)-P_{Z^\perp}\circ P(x^*)\|\\
&\leq \|x^*-P(x^*)\|+\|I-P_{Z^\perp}\|\|P(x^*)\|\leq \|x^*\|.
\end{align*}
Hence, $\|I-P_{Z^\perp}\circ P\|=1$. Now, \cite[Theorem~3.4]{DP} ensures that $P_{Z^\perp}\circ P:X^*\ra Z^\perp$ is the metric projection of norm-$1$. This conclude that $Z$ has property-$(HB)$ in $X$.
\end{proof}

Our next theorem encounters a special case that of Proposition~\ref{TPH1} for the Banach spaces which are of type $L_1$-preduals or $M$-embedded spaces. A Banach space is said to be an $L_1$-predual if its dual is isometric to $L_1(\mu)$ for some measure space $(\Omega, \Sigma, \mu)$.
A Banach space is said to be $M$-embedded if it is an $M$-ideal in its bi-dual.

\begin{thm}
Let $X$ be an $L_1$-predual space or an $M$-embedded space. Let $J$ be an $M$-ideal in $X$ and $Z$ be a finite co-dimensional subspace of $J$ such that $Z$ has property-$(HB)$ in $J$. Then, $Z$ has property-$(HB)$ in $X$.
\end{thm}
\begin{proof}
	Since $J$ is an $M$-ideal in $X$, $J^{\perp\perp}$ is an $M$-summand in $X^{**}$. Let us assume that $X^{**}=J^{\perp\perp}\bigoplus_{\ell_\iy}W$ for some subspace W of $X^{**}$. Now $Z\ci J$ is finite co-dimensional and the space $J$ is a space either an $L_1$-predual or $M$-embedded. Hence, by \cite[Theorem 3.14]{DP}, $Z^{\perp\perp}$ has property-$(HB)$ in $J^{\perp\perp}$. Since $J^{\perp\perp}$ is an $M$-summand in $X^{**}$, by Proposition~\ref{TPH1}, we get $Z^{\perp\perp}$ has property-$(HB)$ in $X^{**}$. Again, by \cite[Theorem 3.14]{DP}, we get $Z$ has property-$(HB)$ in $X$.
\end{proof}

\subsection{Separable determination of these properties.}
In this part, we discuss that these properties are separably determined. First, we have the following observation for property-$(wU)$.
\begin{prop}
Let $Y$ be a subspace of $X$, and every $1$-dimensional subspace of $Y$ have property-$(wU)$ in $X$. Then $Y$ has property-$(wU)$ in $X$. 
\end{prop}
\begin{proof}
Let $y^*\in Y^*$ with $y^*(y_0)=\|y^*\|$ for some $y_0\in Y$ and $x_1^*,x_2^*$ are any two norm-preserving extensions of $y^*$. Put $Y_0=\text{span}\{y_0\}$. Then $x_1^*$ and $x_2^*$ are two norm-preserving extensions of $y^*|_{Y_0}$. Hence, $x_1^*=x_2^*$ as $Y_0$ has propety-$(wU)$ in $X$. Thus, $Y$ has property-$(wU)$ in $X$.
\end{proof}
\begin{cor}
Let $Y$ be a subspace of $X$, and every separable subspace of $Y$ has property-$(wU)$ in $X$. Then $Y$ has property-$(wU)$ in $X$. 
\end{cor}

We now come to property-$(U)$ in Banach spaces. We first observe it is separably determined. The result follows from the following observation by Lima.

\begin{thm}(\cite[Theorem~2.1]{LA})
Let $Y$ be a subspace of $X$, then $Y$ has property-$(U)$ in $X$ if and only if for any $x_1^*, x_2^*\in Y^\#$ with $x_1^*+x_2^*\in Y^\perp$ implies $x_1^*+x_2^*=0$.
\end{thm}

\begin{prop}
Let $Y$ be a subspace of $X$, and every separable subspace of $Y$ has property-$(U)$ in $X$. Then $Y$ has property-$(U)$ in $X$. 
\end{prop}
\begin{proof}
Let $x_1^*,x_2^*\in Y^\#$ and $x_1^*+x_2^*\in Y^\perp$. It remains to prove that $x_1^*+x_2^*=0$. Let $(y_n)\ci S_Y$ such that $\ds\sup_{n\in \mathbb{N}}|x_1^*(y_n)|=\|x_1^*|_Y\|=\|x_1^*\|$ and $(w_n)\ci S_Y$ such that $\ds\sup_{n\in \mathbb{N}}|x_2^*(w_n)|=\|x_2^*|_Y\|=\|x_2^*\|$. Consider $Z=\textrm{span}\{y_n,w_m:n,m\in \mathbb{N}\}$. Clearly, $x_1^*,x_2^*\in Z^\#$ and $x_1^*+x_2^*\in Z^\perp$. Since $Z$ has property-$(U)$ in $X$, by the above, $x_1^*+x_2^*=0$.
\end{proof}

Let us recall the definition of $3.X.I.P.$ for a subspace $Y$.
\bdfn \cite{LA}
Let $Y$ be a subspace of $X$. $Y$ is said to have $n.X.I.P.$ if for any family of $n$ closed balls in $X$ with centres in $Y$, viz. $\{B[y_i,r_i]:1\leq i\leq n\}$ such that $\cap_{i=1}^n B[y_i,r_i]\neq\es$ satisfies $\left(\cap_{i=1}^n B[y_i,r_i+\e]\right)\bigcap Y\neq\es$ for $\e>0$.
\edfn

The following proposition is a direct consequence of the definition of $3.X.I.P.$

\begin{prop}
Let $Y$ be a subspace of $X$, and every $3$-dimensional subspace of $Y$ have $3.X.I.P.$ in $X$. Then $Y$ has $3.X.I.P.$ in $X$. 
\end{prop}

Let us recall that a subspace $Y$ of $X$ with property-$(U)$, satisfies property-$(SU)$ if and only if $Y$ has $3.X.I.P.$ (see \cite[Theorem~ 4.1]{LA}) This concludes our next result.

\begin{prop}\label{SEPSU}
Let $Y$ be a subspace of $X$, and every separable subspace of $Y$ has property-$(SU)$ in $X$. Then $Y$ has property-$(SU)$ in $X$. 
\end{prop}
\begin{prop}
Let $Y$ be a subspace of $X$, and every separable subspace of $Y$ has property-$(HB)$ in $X$. Then $Y$ has property-$(HB)$ in $X$. 
\end{prop}

\begin{proof}
Suppose every separable subspace of $Y$ has property-$(HB)$ in $X$. By Proposition~\ref{SEPSU}, $Y$ has property-$(SU)$ in $X$. We assume that $Y$ does not have property-$(HB)$ in $X$. Let $x^*\in X^*$, $x^*=y^\#+y^\perp$, $y^\#\in Y^\#$ and $y^\perp\in Y^\perp$ such that $\|x^*\|<\|y^\perp\|$. Let $(y_n)\ci S_Y$ such that $\ds\sup_{n\in \mathbb{N}}|y^\#(y_n)|=\|y^\#|_Y\|=\|y^\#\|$. Consider $Z=\textrm{span}\{y_n:n\in \mathbb{N}\}$. Clearly, $y^\#\in Z^\#$, $y^\perp\in Z^\perp$ are related to $x^*=y^\#+y^\perp$, where $\|x^*\|<\|y^\perp\|$. This contradicts that $Z$ has property-$(HB)$ in $X$.
\end{proof}

\section{On Hahn-Banach smoothness and related properties in  sequence spaces}
In this section, we study the finite-dimensional and finite co-dimensional subspaces having all the aforementioned properties in $c_0$, $\ell_p$, where $1<p<\infty$. Most of the results are summarized in two tables for two cases at the end of two respective subsections.
\subsection{In finite-dimensional subspaces}\label{Sec2}

Let us recall \cite[Theorem~3.3]{P}. In this theorem, the author characterizes finite-dimensional subspaces with the property-$(U)$ in $c_0$. We now look at the cases for property-$(SU)$. 

Let us recall that a subspace $Y$ is said to be an ideal in $X$ if $Y^\perp$ is the kernel of a norm-$1$ projection on $X^*$. It is known that $Y$ has property-$(SU)$ in $X$ if and only if $Y$ has property-$(U)$ and it is an ideal in $X$ (see \cite{O}).

As a consequence of \cite[Proposition 2]{RT} and \cite[Corollary, Pg. 50]{LJ}, we have the following.
\begin{prop}\label{H3}
Let $Y$ be a finite-dimensional subspace of $c_0$. Then, $Y$ is an ideal in $c_0$ if and only if $Y$ is isometrically isomorphic to $\ell_\iy(\dim(Y))$.
\end{prop}
Hence, we have the following result:
\begin{thm}\label{H1}
Let $Y$ be a finite-dimensional subspace of $c_0$. Then, $Y$ has property-$(SU)$ in $c_0$ if and only if $Y$ has property-$(U)$ in $c_0$ and it is isometrically isomorphic to $\ell_{\iy}(\dim(Y))$.
\end{thm}
A demonstration of property-$(HB)$ is presented similarly.
\begin{thm}\label{H2}
Let $Y$ be a finite-dimensional subspace of $c_0$. Then, $Y$ has property-$(HB)$ in $c_0$ if and only if $Y$ has property-$(U)$ in $c_0$ and there exists a bi-contractive projection $P:c_0\ra Y$ such that both $Y$ and $(I-P)(c_0)$ are $L_1$-predual spaces.
\end{thm}
\begin{proof}
Suppose that $Y$ has property-$(HB)$ in $c_0$. Hence, $Y$ is an ideal in $c_0$. Clearly, it is an $1$-complemented subspace of $c_0$. Let $P:c_0\ra Y$ be a norm-$1$ projection onto $Y$. It is evident that $\ker(P^*)=Y^\perp$ and $\|P^*\|=1$. As $Y$ has property-$(HB)$ in $c_0$, we have that $\|I-P^*\|=1$. Hence, $\|P\|=\|I-P\|=1$. Now, both $Y$ and $(I-P)(c_0)$ are ranges of norm-$1$ projections. Hence, the result follows from \cite[Corollary, Pg. 50]{LJ}.

For the converse part, let $P:c_0\ra Y$ be a bi-contractive projection with range$(P)=Y$. Clearly, $\|P^*\|=1=\|I-P^*\|$. Since $Y$ has property-$(U)$ in $c_0$ and $I-P^*$ is a bi-contractive projection from $\ell_1$ onto $Y^\perp$, by \cite[Theorem~3.4]{DP}, $Y$ has property-$(HB)$ in $c_0$.
\end{proof}

The finite-dimensional subspaces of $\ell_1$ having property-$(U)$ are precisely discussed in \cite[Theorem~2.3]{P}. In the same space, let's focus on the properties-$(SU)$ and property-$(HB)$.
\begin{thm}\label{H8}
Let $Y$ be a finite-dimensional subspace of $\ell_1$. Then, $Y$ has property-$(SU)$ in $\ell_1$ if and only if $Y$ has property-$(U)$ in $\ell_1$ and $Y$ is isometrically isomorphic to $\ell_1(\dim (Y))$.
\end{thm}
\begin{proof}
Since $Y$ is finite-dimensional subspace of $\ell_1$ and it is an ideal, $Y=Y^{\perp\perp}$ is the range of a norm-$1$ projection (see \cite[Pg. 597]{RT}). Hence, by \cite[Chapter~6, Theorem~3, Pg. 162]{L}, it is isometrically isomorphic to $L_1(\mu)$ for some measure $\mu$. As the dimension of $Y$ is finite, $Y$ is isometrically isomorphic to $\ell_1(\dim (Y))$.

For the converse part, it follows from \cite[Chapter~6, Theorem~3, Pg. 162]{L} that $Y$ is the range of a norm-$1$ projection on $\ell_1$. Hence, it is an ideal in $\ell_1$ and the result follows.
\end{proof}

\begin{thm}\label{H100}
Let $Y$ be a finite-dimensional subspace of $\ell_1$. Then, $Y$ has property-$(HB)$ in $\ell_1$ if and only if $Y$ has property-$(U)$ in $\ell_1$ and there exists a bi-contractive projection $P:\ell_1\ra Y$ such that both $Y$, $(I-P)(\ell_1)$ are isometrically isomorphic to $\ell_1(\dim(Y))$ and $\ell_1$, respectively.
\end{thm}
\begin{proof}
Suppose that $Y$ has property-$(HB)$ in $\ell_1$. Since $Y$ is an ideal in $X$, $Y=Y^{\perp\perp}$ is the range of a norm-$1$ projection. Using similar arguments to those used in Theorem~\ref{H2}, we have that both $Y$ and $(I-P)(\ell_1)$ are ranges of norm-$1$ projections. Hence, the result follows from \cite[Chapter~6, Theorem~3, Pg. 162]{L}.

Using similar arguments to those used to prove the converse of Theorem~\ref{H2}, one can prove the converse.
\end{proof}

We now come for the spaces $\ell_p$. Due to the strict convexity of $\ell_q$, every subspace of $\ell_p$ has property-$(U)$, where $\frac{1}{p}+\frac{1}{q}=1$, $1<p<\infty$, $p\neq 2$. 

We have the following results from \cite[Theorem~1.2]{RB} as  every ideal in $\ell_p$ is an $1$-complemented subspace (see \cite[Proposition~2]{RT}).
\begin{thm}\label{H6}
Let $Y$ be a subspace of $\ell_p$. Then, $Y$ is an ideal in $\ell_p$ if and only if 
\bln
\item $Y$ is isometrically isomorphic to $\ell_{p}(\dim (Y))$.

or

\item $Y$ is spanned by a family of elements of disjoint supports.
\el 
\end{thm}
Since every subspace of $\ell_p$ has property-$(U)$, we have the following.
\begin{cor}\label{H7}
Let $Y$ be a finite-dimensional subspace of $\ell_p$. Then, $Y$ has property-$(SU)$ in $\ell_p$ if and only if $Y$ is isometrically isomorphic to $\ell_p(\dim (Y))$.
\end{cor}
The following result is guaranteed by a demonstration similar to that of Theorem~\ref{H100}.

\begin{thm}\label{I4}
Let $Y$ be a finite-dimensional subspace of $\ell_p$. Then, $Y$ has property-$(HB)$ in $\ell_p$ if and only if there exists a bi-contractive projection $P:\ell_p\ra Y$ such that both $Y$ and $(I-P)(\ell_p)$ are isometrically isomorphic to $\ell_p(\dim(Y))$ and $\ell_p$, respectively.
\end{thm}
We also have the following.
\begin{thm}\label{H5}
Let $Y=\textrm{span} \{a^1,a^2,\dots,a^n\}$, $a^1,a^2,\dots,a^n\in \ell_p$. Then, $Y$ has property-$(HB)$ in $\ell_p$ if and only if there exist $n$ different indices $t_1,t_2,\dots,t_n$ such that:
\bla
\item $a^i_{t_j}=\delta_{i,j}$ ($1\leq i,j\leq n$)
\item for any $i$, there exists at most one index $s_i\notin \{t_1,t_2,\dots,t_n\}$ such that $|a^i_{s_i}|=1$ and $a^j_{s_i}=0$ for any $j\neq i$.
\el
\end{thm}
\begin{proof}
Suppose that $Y$ has property-$(HB)$ in $X$. Hence, $Y^\perp$ is the range of a bi-contractive projection. As $Y^\perp=\bigcap_{i=1}^{n}\ker J(a_i)$, where $J$ is the canonical embedding map. Thus, the result follows from \cite[Theorem 8]{BP2}.

  Conversely, it follows from \cite[Theorem 8]{BP2} that $Y^\perp$ is a range of a bi-contractive projection. Since $Y$ has property-$(U)$ in $\ell_p$, the result follows from \cite[Theorem~3.4]{DP}.
\end{proof}

We now discuss the space $\ell_2$. Let $Y$ be a closed subspace of $\ell_2$ and $Y^\perp$ be the orthogonal complement of $Y$. Clearly, for $x\in \ell_2$, $x=y+y^\perp$, $y\in Y$ and $y^\perp\in Y^\perp$, we have that $\|x\|^2=\|y\|^2+\|y^\perp\|^2$. Hence, every subspace of $\ell_2$ has property-$(U)$ ($(SU)$, $(HB)$).

\textit{The above results can now be summarized in the table below}. 

\vspace{4mm}
 \hskip-0.4cm\begin{tabular}{|p{2cm}|p{3cm}|p{3cm}|p{3cm}|}
 \hline
 Space & property-$(U)$ & property-$(SU)$ & property-$(HB)$ \\
 \hline
 $c_0$  & \cite[Theorem~3.3]{P} &  Theorem \ref{H1} & Theorem~\ref{H2}  \\
 \hline
 $\ell_1$   & \cite[Theorem~2.3]{P} & Theorem~\ref{H8} &Theorem~\ref{H100} \\
\hline
 $\ell_p$, 
 \vspace{2mm}\tiny$1<p<\infty$, 
 \vspace{1mm}$p\neq 2$  & Every subspace has this property & Corollary~\ref{H7} & Theorem~\ref{H5}, Theorem~\ref{I4}\\
 \hline
\end{tabular}
\vspace{4mm}

\subsection{In finite co-dimensional subspaces}\label{Sec3}
We continue our discussion for the finite co-dimensional subspaces in the aforementioned sequence spaces. 

First, we discuss in $c_0$. The authors in \cite{DPR} characterize the finite co-dimensional proximinal subspaces having property-$(U)$ in $c_0$ (see \cite[Proposition~6.1]{DPR}). A similar case for property-$(SU)$ is also mentioned in \cite{DPR} as stated below.
\begin{thm}\label{S1}\cite{DPR}
Let $Y\ci c_0$ be a subspace of finite co-dimension. If $Y$ is proximinal and has property-$(SU)$ in $c_0$, then $Y$ can be decomposed in $c_0$ as $Y=F\bigoplus_{\ell_\iy}Z$, where $F$ has property-$(SU)$ in $\ell_\infty(k)$ for some $k\geq \dim(c_0 /Y)$, it is isometric to $\ell_\iy(n)$ for some $n\in \mathbb{N}$ and $Z$ is an $M$-summand in $c_0$. Conversely, if $Y=F\bigoplus_{\ell_\iy}Z$ , where $F$ has property-$(SU)$ in $\ell_\iy(k)$, $Z$ is an $M$-summand in $c_0$, then $Y$ has property-$(SU)$ in $c_0$. 
\end{thm}
We have the following for property-$(HB)$.
\begin{thm}\label{S2}
Let $Y$ be a subspace of finite co-dimension in $c_0$. Then, $Y$ has property-$(HB)$ in $c_0$ if and only if $Y=F\bigoplus_{\ell_\infty} Z$, where $F\ci \ell_\infty (k)$ for some $k\in \mathbb{N}$, $Z=\{z \in c_0: z(j) = 0 ~\mbox{for all}~ j \leq k\}$ and $F$ has property-$(HB)$ in $\ell_\infty (k)$.
\end{thm}
\begin{proof}
Let $Y$ be a subspace of finite co-dimension in $c_0$. Then $Y=F\bigoplus_{\ell_\infty} Z$, where $F\ci \ell_\infty (k)$ for some $k\in \mathbb{N}$ and $Z=\{z \in c_0: z(j) = 0 ~\mbox{for all}~ j \leq k\}$ (see \cite[Proposition~6.1]{DPR}).

Suppose $Y$ has property-$(HB)$ in $c_0$. By Theorem~\ref{S1}, $F$ has property-$(SU)$ in $\ell_\iy(k)$. From \cite[Theorem~3.3]{DP}, we have $P_{F^\perp}:\ell_1(k)\ra \ell_1(k)$ is a linear projection. Now, there exists a unique linear Hahn-Banach extension operator $T:Y^*\ra \ell_1$ i.e., $T(y^*)$ is the unique norm preserving extension of $y^*\in Y^*$. As $Y=F\bigoplus_{\ell_\infty} Z$, we have $T:F^*\bigoplus_{\ell_1} Z^*\ra \ell_1(k)\bigoplus_{\ell_1} Z^*$. Consider $T_1=T|_{F^*}$, then $T_1:F^*\ra \ell_1(k)$ is the Hahn-Banach extension map from $F^*$ to $\ell_1(k)$. One can check that $T^*_1:\ell_{\iy}(k)\ra F$ is a norm-$1$ projection. So, we have $\ell_{\iy}(k)=F\bigoplus G$ for some subspace $G$ of $\ell_{\iy}(k)$. It is clear that  $(I-T_1^*)^*:\ell_1(k)\ra \ell_1(k)$ is a projection and $(I-T_1^*)^*=P_{F^{\perp}}$. If $Y$ has property-$(HB)$ in $c_0$, then $T_1^*$ is a bi-contractive projection, which in other words $\|P_{F^\perp}\|=1$. Hence, by \cite[Theorem~3.4]{DP}, it follows that $F$ has property-$(HB)$ in $\ell_{\iy}(k)$.

Conversely, suppose that $F$ has property-$(HB)$ in $\ell_{\iy}(k)$ and $P:\ell_{\iy}(k)\ra F$ be the bi-contractive projection. We define $Q:\ell_{\iy}(k) \bigoplus_{\ell_\iy} Z\ra Y$  by $Q((a_n)_{n=1}^{k})=P((a_n)_{n=1}^{k})$ for $(a_n)_{n=1}^{k}\in \ell_{\iy}(k)$ and $Q(z)=z$ for $z\in Z$. It is evident that $Q$ is a bi-contractive projection from $c_0$ onto $Y$. Hence, $Y$ has property-$(HB)$ in $c_0$.
\end{proof}

The discussion will now take place in the $\ell_1$ space.
\begin{prop}\label{S6} \cite[Pg. 251]{P}
Let $(a_n)\in \ell_\iy$ and $Y=\ker (a_n)$. Then, $Y$ has property-$(U)$ in $\ell_1$ if and only if $|a_n|>0$ for all $n\in \mathbb{N}$.
\end{prop}

\begin{rem} \label{S7}
It is also discussed in \cite{P} that there is no subspace of co-dimension greater than one having property-$(U)$ in $\ell_1$. Hence, there is no subspace of co-dimension greater than one having property-$(SU)$, ($(HB)$) in $\ell_1$
\end{rem}
The situation remains unchanged when the co-dimension is one.

\begin{thm}\label{I1}
Let $Y$ be a subspace of co-dimension $1$ in $\ell_1$. Then $Y$ does not have property-$(SU)$ in $\ell_1$.
\end{thm}
\begin{proof}
Suppose that $Y$ has property-$(SU)$ in $\ell_1$. Then there exists a norm-$1$ projection $P:\ell_{\iy}\ra Y^\#$ such that $\|P((a_n))\|<\|(a_n)\|$ whenever $P((a_n))\neq (a_n)$, $(a_n)\in \ell_\infty$. Again, as $Y$ has property-$(SU)$ in $\ell_1$, we have that $\ell_\infty=Y^{\perp}\bigoplus Y^\#$. Since $Y^\perp$ is an $1$-dimensional subspace, $Y^\#$ is an one co-dimensional subspace. Hence, the result follows from \cite[Theorem~1.1]{B}.
\end{proof}

We now come for the spaces $\ell_p$, $1<p<\iy$, $p\neq 2$.
By \cite[Proposition~2]{RT}, every ideal in $\ell_p$ is $1$-complemented subspace. Hence, the following consequence follows from  \cite[Theorem~3.4]{BP1}.
\begin{prop}\label{S4}
Let $Y$ be a subspace of co-dimension $n$ in $\ell_p$. Then, $Y$ is an ideal in $\ell_p$ if and only if there exists a basis $\{a^1,a^2,\dots,a^n\}$ for $Y^\perp$ such that every $a^i$ has at most two non-null components for $1\leq i\leq n$. 
\end{prop}
Since every subspace of $\ell_p$ has property-$(U)$, we have the following.
\begin{thm}\label{S5}
Let $Y$ be a subspace of co-dimension $n$ in $\ell_p$. Then, $Y$ has property-$(SU)$ in $\ell_p$ if and only if there exists a basis $\{a^1,a^2,\dots,a^n\}$ for $Y^\perp$ such that every $a^i$ has at most two non-null components for $1\leq i\leq n$.
\end{thm}
In the case of property-$(HB)$, we have the following.
\begin{thm}\label{HB1}
Let $Y$ be a subspace of co-dimension $n$ in $\ell_p$. Then, $Y$ has property-$(HB)$ in $\ell_p$ if and only if  there exist $n$ elements $a^1,a^2,\dots,a^n$ in $\ell_q$, $(\frac{1}{p}+\frac{1}{q}=1)$ and $n$ different indices
$t_1,t_2,\dots,t_n$ such that
\bla
\item $Y=\bigcap_{i=1}^{n}\ker (a^i)$.
\item $(a^i)_{t_j}=\delta_{ij}$, $1\leq i,j \leq n$.
\item for every $i\in \mathbb{N}$ , there exists at most one index $s_i\notin \{t_1,t_2,\dots,t_n\}$ such that $|(a^i)_{s_i}| =1$ and $(a^j)_{s_i}= 0$ for $j\neq i$.
\el 
\end{thm}
\begin{proof}
Suppose that $Y$ has property-$(HB)$ in $\ell_p$. Hence, $Y$ is an ideal, and by \cite[Proposition 2]{RT}, $Y$ is the range of a norm-$1$ projection. Let $P:\ell_p\ra Y$ be a projection. Using similar arguments used in the proof in Theorem~\ref{H2}, we have that $\|I-P\|=1$. Thus, $Y$ is the range of a bi-contractive projection. Therefore, the result follows from \cite[Theorem~8]{BP2}.

Conversely, by \cite[Theorem~8]{BP2}, we have that $Y$ is a range of a bi-contractive projection, say, $P:\ell_p\ra Y$. Then $I-P^*:\ell_p\ra Y^\perp$ is a bi-contractive projection. Hence, the results from \cite[Theorem~3.4]{DP}.
\end{proof}

{\it The above results can now be summarized in the table below}. 

\vspace{4mm}
 \hskip-0.4cm\begin{tabular}{|p{1cm}|p{3.5cm}|p{3.5cm}|p{3.5cm}|}
 \hline
 Space & property-$(U)$ & property-$(SU)$ & property-$(HB)$ \\
 \hline
 $c_0$  & See \cite[Proposition~6.1]{DPR}

 & $\bullet$ For co-dim=1, see \cite[Theorem~6.8]{DPR}
 
 \vspace{4mm}
 $\bullet$ For co-dim $>1$, see Theorem~\ref{S1}.
 & $\bullet$ For co-dim=1, see \cite[Corollary~3.20]{DP}
 
 \vspace{4mm}
 $\bullet$ For co-dim$>1$, see Theorem~\ref{S2}\\
 \hline
 $\ell_1$   & $\bullet$ For co-dim$=1$, see Proposition~\ref{S6}.
 
 \vspace{4mm}
 $\bullet$ None of co-dim$>1$, see Remark~\ref{S7}.
 & $\bullet$  None of co-dim$=1$, see Theorem~\ref{I1}.
 
 \vspace{4mm}
 $\bullet$ None of co-dim$>1$, see Remark~\ref{S7}. 
 & $\bullet$  None of co-dim$=1$, see Theorem~\ref{I1}.
 
 \vspace{4mm}
 $\bullet$ None of co-dim$>1$, see Remark~\ref{S7}. \\
\hline
 $\ell_p$, 
 
 \vspace{2mm}\tiny$1<p<\infty$,
 
 $p\neq 2$
 & Every subspace has this property & Theorem~\ref{S5} & Theorem~\ref{HB1}\\
 \hline
\end{tabular}

\bibliographystyle{plain, abbrv}

\end{document}